\documentclass[reqno, 12pt]{amsart}
\pdfoutput=1
\makeatletter
\let\origsection=\section \def\section{\@ifstar{\origsection*}{\mysection}}
\def\mysection{\@startsection{section}{1}\z@{.7\linespacing\@plus\linespacing}{.5\linespacing}{\normalfont\scshape\centering\S}}
\makeatother

\usepackage{amsmath,amssymb,amsthm}
\usepackage{mathrsfs}
\usepackage{mathabx}\changenotsign
\usepackage{dsfont}
\usepackage{bbm}

\usepackage{xcolor}
\usepackage[backref]{hyperref}
\hypersetup{
	colorlinks,
	linkcolor={red!60!black},
	citecolor={green!60!black},
	urlcolor={blue!60!black}
}

\usepackage[open,openlevel=2,atend]{bookmark}

\usepackage[abbrev,msc-links,backrefs]{amsrefs}
\usepackage{doi}

\renewcommand{\PrintDOI}[1]{\doi{#1}}

\usepackage[T1]{fontenc}
\usepackage{lmodern}
\usepackage[babel]{microtype}
\usepackage[english]{babel}
\usepackage{verbatim}

\usepackage{geometry}
\numberwithin{equation}{section}
\numberwithin{figure}{section}

\usepackage{enumitem}

\let\polishlcross=\l
\def\l{\ifmmode\ell\else\polishlcross\fi}

\def\qand{\quad\text{and}\quad}

\let\emptyset=\varnothing
\let\setminus=\smallsetminus

\makeatletter
\def\moverlay{\mathpalette\mov@rlay}
\def\mov@rlay#1#2{\leavevmode\vtop{   \baselineskip\z@skip \lineskiplimit-\maxdimen
		\ialign{\hfil$\m@th#1##$\hfil\cr#2\crcr}}}
\newcommand{\charfusion}[3][\mathord]{
	#1{\ifx#1\mathop\vphantom{#2}\fi
		\mathpalette\mov@rlay{#2\cr#3}
	}
	\ifx#1\mathop\expandafter\displaylimits\fi}
\makeatother

\newcommand{\bigdcup}{\charfusion[\mathop]{\bigcup}{\cdot}}

\DeclareFontFamily{U}  {MnSymbolC}{}
\DeclareSymbolFont{MnSyC}         {U}  {MnSymbolC}{m}{n}
\DeclareFontShape{U}{MnSymbolC}{m}{n}{
	<-6>  MnSymbolC5
	<6-7>  MnSymbolC6
	<7-8>  MnSymbolC7
	<8-9>  MnSymbolC8
	<9-10> MnSymbolC9
	<10-12> MnSymbolC10
	<12->   MnSymbolC12}{}
\DeclareMathSymbol{\powerset}{\mathord}{MnSyC}{180}

\usepackage{tikz}
\usetikzlibrary{calc,decorations.pathmorphing}
\pgfdeclarelayer{background}
\pgfdeclarelayer{foreground}
\pgfdeclarelayer{front}
\pgfsetlayers{background,main,foreground,front}

\usepackage{subcaption}
\newcommand{\qedge}[7]{
	
	\ifx\relax#4\relax
	\def\qoffs{0pt}
	\else
	\def\qoffs{#4}
	\fi
	
	\def\qhedge{
		($#1+#3!\qoffs!-90:#2-#3$) --
		($#2+#1!\qoffs!-90:#3-#1$) --
		($#3+#2!\qoffs!-90:#1-#2$) -- cycle}

	\coordinate (12) at ($#1!\qoffs!90:#2$);
	\coordinate (13) at ($#1!\qoffs!-90:#3$);
	\coordinate (23) at ($#2!\qoffs!90:#3$);
	\coordinate (21) at ($#2!\qoffs!-90:#1$);
	\coordinate (31) at ($#3!\qoffs!90:#1$);
	\coordinate (32) at ($#3!\qoffs!-90:#2$);
	
	\def\nqhedge{
		(13) let \p1=($(13)-#1$), \p2=($(12)-#1$) in
		arc[start angle={atan2(\y1,\x1)}, delta angle={atan2(\y2,\x2)-atan2(\y1,\x1)-360*(atan2(\y2,\x2)-atan2(\y1,\x1)>0)}, x radius=\qoffs, y radius=\qoffs] --
		(21) let \p1=($(21)-#2$), \p2=($(23)-#2$) in
		arc[start angle={atan2(\y1,\x1)}, delta angle={atan2(\y2,\x2)-atan2(\y1,\x1)-360*(atan2(\y2,\x2)-atan2(\y1,\x1)>0)}, x radius=\qoffs, y radius=\qoffs] --
		(32) let \p1=($(32)-#3$), \p2=($(31)-#3$) in
		arc[start angle={atan2(\y1,\x1)}, delta angle={atan2(\y2,\x2)-atan2(\y1,\x1)-360*(atan2(\y2,\x2)-atan2(\y1,\x1)>0)}, x radius=\qoffs, y radius=\qoffs] --
		cycle}
	
	\ifx\relax#5\relax
	\def\qlwidth{1pt}
	\else
	\def\qlwidth{#5}
	\fi
	
	\ifx\relax#7\relax
	\fill \nqhedge;
	\else
	\fill[#7]\nqhedge;
	\fi
	
	\ifx\relax#6\relax
	\draw[dotted, line width=\qlwidth,rounded corners=\qoffs]\nqhedge;
	\else
	\draw[dotted, line width=\qlwidth,#6]\nqhedge;
	\fi
}

\let\epsilon=\varepsilon

\let\rho=\varrho
\let\theta=\vartheta

\newcommand{\ccP}{\mathscr{P}}

\newtheoremstyle{note}  {4pt}  {4pt}  {\sl}  {}  {\bfseries}  {.}  {.5em}          {}
\newtheoremstyle{introthms}  {3pt}  {3pt}  {\itshape}  {}  {\bfseries}  {.}  {.5em}          {\thmnote{#3}}
\newtheoremstyle{remark}  {2pt}  {2pt}  {\rm}  {}  {\bfseries}  {.}  {.3em}          {}

\theoremstyle{plain}
\newtheorem{theorem}{Theorem}[section]
\newtheorem{lemma}[theorem]{Lemma}

\newtheorem{prop}[theorem]{Proposition}

\newtheorem{cor}[theorem]{Corollary}
\newtheorem{fact}[theorem]{Fact}
\newtheorem{claim}[theorem]{Claim}

\theoremstyle{note}
\newtheorem{dfn}[theorem]{Definition}

\theoremstyle{remark}
\newtheorem{remark}[theorem]{Remark}

\usepackage{accents}

\usepackage{lineno}
\newcommand*\patchAmsMathEnvironmentForLineno[1]{
	\expandafter\let\csname old#1\expandafter\endcsname\csname #1\endcsname
	\expandafter\let\csname old#1\expandafter\endcsname\csname end#1\endcsname
	\renewenvironment{#1}
	{\linenomath\csname old#1\endcsname}
	{\csname oldend#1\endcsname\endlinenomath}}

\AtBeginDocument{
}

\def\NIP{necessary intersection point}
\def\secmin{\rm{secmin~}}
\newcommand{\blockfam}[2]{\mathcal{B}(#1,#2)}
\newcommand{\scatfam}[3]{\mathcal{A}(#1,#2,#3)}

\begin{document}
	
	\title[$r$-cross~$t$-intersecting families via necessary intersection points]{$r$-cross~$t$-intersecting families via necessary intersection points}
	\author[P. Gupta]{Pranshu Gupta}
	\address{Hamburg University of Technology, Institute of Mathematics, Hamburg, Germany}
	\email{pranshu.gupta@tuhh.de}
	\author[Y. Mogge]{Yannick Mogge}
	\address{Hamburg University of Technology, Institute of Mathematics, Hamburg, Germany}
	\email{yannick.mogge@tuhh.de}
	\author[S. Piga]{Sim\'on Piga}
	\address{Fachbereich Mathematik, Universit\"{a}t Hamburg, Hamburg, Germany}
	\email{simon.piga@uni-hamburg.de}
	\author[B.~Sch\"{u}lke]{Bjarne Sch\"{u}lke}
	\address{Fachbereich Mathematik, Universit\"{a}t Hamburg, Hamburg, Germany}
	\email{bjarne.schuelke@uni-hamburg.de}
	\thanks{The third author was supported by ANID/CONICYT Acuerdo Bilateral DAAD/62170017 through a Ph.D. Scholarship.}

	
	\begin{abstract}
		Given integers~$r\geq 2$ and~$n,t\geq 1$ we call families~$\mathcal{F}_1,\dots,\mathcal{F}_r\subseteq\ccP([n])$ $r$-cross~$t$-intersecting if for all~$F_i\in\mathcal{F}_i$,~$i\in[r]$, we have~$\vert\bigcap_{i\in[r]}F_i\vert\geq t$.
		We obtain a strong generalisation of the classic Hilton-Milner theorem on cross intersecting families.
		In particular, we determine the maximum of~$\sum_{j\in [r]}\vert\mathcal{F}_j\vert$ for~$r$-cross~$t$-intersecting families in the cases when these are~$k$-uniform families or arbitrary subfamilies of~$\ccP([n])$.
		Only some special cases of these results had been proved before.
		We obtain the aforementioned theorems as instances of a more general result that considers measures of~$r$-cross $t$-intersecting families.
		This also provides the maximum of~$\sum_{j\in [r]}\vert\mathcal{F}_j\vert$ for families of possibly mixed uniformities~$k_1,\ldots,k_r$. 
		
%



	\end{abstract}

	\maketitle\section{Introduction}
	
	One of the main themes in extremal set theory are intersecting families.
	Given some~$n\in\mathds{N}$ a family~$\mathcal{F}\subseteq\ccP([n])$ is said to be \emph{intersecting} if for all~$F,F'\in\mathcal{F}$ we have~$F\cap F'\neq\emptyset$.
	The following well-known theorem by Erd\H{o}s, Ko, and Rado~\cite{EKR} is one of the earliest results in extremal set theory.
	
	\begin{theorem}\label{thm:EKR}
		Let~$k,n\in\mathds{N}$ with~$2k\leq n$ and let~$\mathcal{F}\subseteq [n]^{(k)}$ be an intersecting family.
		Then~$\vert\mathcal{F}\vert\leq\binom{n-1}{k-1}$ and this bound is sharp.
	\end{theorem}
	
	Observe that this maximum is attained by a family which contains all the sets of size~$k$ that contain one fixed element, for instance~$\mathcal F = \{F\in [n]^{(k)} \colon 1\in F\}$. 
	
	As a variation \emph{cross intersecting} families can be considered.
	For~$r,t,n\in\mathds{N}$ we say that families~$\mathcal{F}_1,\dots,\mathcal{F}_r\subseteq\ccP([n])$ are~\emph{$r$-cross~$t$-intersecting} if for all~$F_1\in\mathcal{F}_1,\ldots,F_r\in\mathcal{F}_r$ we have~$\vert \bigcap_{i\in[r]}F_i\vert\geq t$.
	Now it is natural to ask for the maximum of~$\sum_{i\in[r]}\vert\mathcal{F}_i\vert$ taken over all non-empty~$r$-cross~$t$-intersecting families~$\mathcal{F}_1,\ldots,\mathcal{F}_r$.
	In this regime there are several partial results  concerning the maximal sizes of~$r$-cross~$t$-intersecting families for specific instances of~$r$ and~$t$, starting with theorems by Hilton~\cite{Hilton1} and by Hilton and Milner~\cite{HiltonMilner} and continued, for instance, in~\cite{FranklandTokushigeNonuniform, FranklKupavskii, FranklWong, ShiFranklQian, SudaTanaka} (also see the references therein). We determine~$\sum_{i\in[r]}\vert\mathcal{F}_i\vert$ for every~$r \geq 2 $ and~$t \geq 1 $ for both uniform families and non-uniform families (see Corollary~\ref{cor:uniform} and Corollary~\ref{cor:nonuniform}), solving several problems posed by Shi, Frankl, Qian~\cite{ShiFranklQian} and generalising a result by Frankl and Wong H.W.~\cite{FranklWong}.
	
	In fact, we show these results in the more general setting of measures.
	Here, one can ask for the maximum measure of families instead of their sizes. 
	More formally, consider a function~$\mu:\ccP([n])\to\mathds{R}_{\geq 0}$, that is, a map assigning a weight to each set in~$\ccP([n])$.
	Now, instead of asking for the maximal size of an intersecting family, we ask for the maximal measure of an intersecting family, where the measure of a family~$\mathcal{F}\subseteq\ccP([n])$ is defined as~$\mu(\mathcal{F})=\sum_{F\in\mathcal{F}}\mu(F)$.
	Two commonly considered measures are the product measure $\varrho_p$ and the uniform measure $\nu_k$.
	For~$p\in[0,1]$ we define the product measure as~$\varrho_p(F)=p^{\vert F\vert}(1-p)^{n-\vert F\vert}$, where~$F\in \ccP([n])$.
	Note that this can be interpreted as the probability that a specific set~$F$ is the result of a random experiment which includes each element from~$[n]$ with probability~$p$ in~$F$.
	The uniform measure~$\nu_k$, with~$k\in [n]$, is defined as~$\nu_k(F)=1/\binom{n}{k}$ if~$\vert F\vert=k$ and~$\nu_k(F)=0$ if~$\vert F\vert\neq k$.
	
	For these measures analogues of the Erd\H{o}s-Ko-Rado theorem can be considered. 
	Indeed, we can reformulate Theorem~\ref{thm:EKR} as follows: For~$k,n\in\mathds{N}$ with~$2k\leq n$ and an intersecting family~$\mathcal{F}\subseteq[n]^{(k)}$ it follows that~$\nu_k(\mathcal{F})\leq\frac{k}{n}$.
	For the product measure the following analogous result was first proved in \cite{productEKR}.
	For~$p\leq 1/2$ and an intersecting family~$\mathcal{F}\subseteq\ccP([n])$ we have~$\varrho_p(\mathcal{F})\leq p$. 
	
	Several results for specific measures and (cross) intersecting families are known, see \cite{AK3,BeyEngel,TokushigeproductAK,DinurSafra,Filmus1,Filmus2}.
	For a more thorough overview we 
	recommend Chapter~12 in~\cite{FranklTokushige}.
	In particular, a result due to Borg~\cite{Borgcrosstintprod} determines the maximum product of measures of~$2$-cross~$t$-intersecting families.
	
	Since normally the measures considered depend only on the size of the sets, we will introduce the following abuse of notation.
	For a function~$\mu:[n]_0\to\mathds{R}_{\geq 0}$ and a set~$F\subseteq [n]$ we consider the measure~$F\mapsto\mu(\vert F\vert)$ but omit the vertical lines, i.e., we write~$\mu(F)$ instead of~$\mu(\vert F\vert)$ and we refer to~$\mu$ as a measure. 
	Further, for~$\mathcal{F}\subseteq\ccP([n])$ we will write~$\mu(\mathcal{F})=\sum_{F\in\mathcal{F}}\mu(F)$.
	
	In our main results we determine the maximum sum of measures of~$r$-cross~$t$-intersecting families.
	Given~$n, a, t \in \mathbb N$ with~$n\geq a\geq t$ consider the families
	\begin{align*}
		\scatfam{n}{a}{t} &= \{F\in \ccP([n]) \colon \vert F\cap [a]\vert \geq t\} \\ 
		\blockfam{n}{a} &= \{F\in \ccP([n]) \colon [a]\subseteq F\}
	\end{align*}
	Essentially, our main results states that the maximum is attained by families ``derived'' from~$\scatfam{n}{a}{t}$ and~$\blockfam{n}{a}$, even when we consider different kinds of measures (including~$\nu_k$ and~$\varrho_p$ when~$p\leq 1/2$).
	Given a set~$A$ we write~$A^{(k)}$ for the set of~$k$-element subsets of~$A$ and similarly~$A^{(\leq k)}$ for the set containing all subsets of~$A$ that are of size at most~$k$.
	Further, for~$\mathcal{F}\subseteq \ccP([n])$ and~$k\in\mathds{N}$ we set~$\mathcal{F}^k=\{F \in \mathcal{F}: \vert F\vert=k\}$ and~$\mathcal{F}^{\leq k}=\{F \in \mathcal{F}: \vert F\vert \leq k\}$.
	Let~$r$ be an integer with~$r\geq 2$, for every~$i\in[r]$ let~$k_i\in\mathds{N}$, and let~$j\in[r]$ such that~$k_j=\min_{i\in[r]}k_i$.
	Then we write~$\underset{i\in[r]}{\secmin}k_i=\displaystyle\min_{i\in[r]\setminus j}k_i$.
	Let us now state our first result which in particular determines the maximum of~$\sum_{i\in[r]}\vert\mathcal{F}_i\vert$ for~$k$-uniform families.
	
	\begin{theorem}\label{thm: main1}
		Let~$r,t,n\in\mathds{N}$ with~$r\geq 2$.
		For~$i\in[r]$ let~$k_i\in[n]$,~$\mu_i:[n]_0\to\mathds{R}_{\geq 0}$, and~$\mathcal{F}_i\subseteq[n]^{(\leq k_i)}$ such that~$n\geq 2\displaystyle\max_{i\in[r]}k_i+\underset{i\in[r]}{\secmin}k_i-t$.
		If~$\mathcal{F}_1,\dots, \mathcal{F}_r$ are non-empty~$r$-cross~$t$-intersecting families, then
		\begin{align}\label{eq:maxfam}
			\sum_{j\in [r]}\mu_j(\mathcal{F}_j)\leq \max\Big\{\mu_{\ell}(\scatfam{n}{a}{t}^{\leq k_{\ell}})+\sum_{j\in[r]\setminus\ell}\mu_j(\blockfam{n}{a}^{\leq k_j})\Big\}\,,
		\end{align}
		where the maximum is taken over~$\ell\in[r]$ and~$a\in\Big[t,\displaystyle\min_{i\in[r]\setminus\ell}k_i\Big]$.
	\end{theorem}
	Note that $\scatfam{n}{i}{t}$ together with $r-1$ copies of $\blockfam{n}{i}$ are $r$-cross $t$-intersecting for every $i\geq t$.
	Thus, this result is sharp in the sense that there are~$r$-cross~$t$-intersecting families which attain the bound.
	
	As mentioned above, applying Theorem~\ref{thm: main1} with~$k_i=k$ and the measure~$\mu_i = \nu_k\binom{n}{k}$ for every~$i\in [r]$, we obtain the following result for~$k$-uniform families.
	\begin{cor}\label{cor:uniform}
		Let~$r\geq 2$, and~$n,t\geq 1$ be integers,~$k\in[n]$, and for~$i\in[r]$ let~$\mathcal{F}_i\subseteq [n]^{(k)}$.
		If~$\mathcal{F}_1,\dots,\mathcal{F}_r$ are non-empty~$r$-cross~$t$-intersecting families and~$n\geq 3k-t$, then 
		$$\sum_{j\in [r]}\vert\mathcal{F}_j\vert
		\leq 
		\max_{m\in[t,k]}\Big\{ \sum_{i\in[t,k]}{m \choose {i}} \cdot {{n-m} \choose {k-i}} +(r-1){{n-m} \choose {k-m}}\Big\}\, $$ and this bound is attained.
	\end{cor}
	
	Some special cases of this result were obtained before.
	For~$r=2$ and~$t\geq 1$ Corollary~\ref{cor:uniform} was proved by Frankl and Kupavskii~\cite{FranklKupavskii}.
	For~$t=1$ and~$r \geq 2$ Corollary~\ref{cor:uniform} was shown very recently in independent work by Shi, Frankl, and Qian \cite{ShiFranklQian}, where they deduce it from a result about two families by an elegant application of the Kruskal-Katona theorem.
	
	Note that, in fact, Theorem~\ref{thm: main1} also determines the maximum of~$\sum_{i\in[r]}\vert\mathcal{F}_i\vert$ for families of different uniformities~$k_1,\ldots,k_r$.
	This solves a problem posed by Shi, Frankl, and Qian~\cite{ShiFranklQian} (apart from the range~$2\displaystyle\max_{i\in[r]}k_i-t<n<2\displaystyle\max_{i\in[r]}k_i+\underset{i\in[r]}{\secmin}k_i-t$).
	
	In the context of non-uniform families, one of the results of a very recent work by Frankl and Wong H.W.~\cite{FranklWong} establishes the maximum possible size of~$2$-cross~$t$-intersecting families.
	The following theorem generalises their result for all~$r\geq 2$ and for measures.
	\begin{theorem}\label{thm: main2}
		Let~$r,t,n\in\mathds{N}$ with~$r\geq 2$.
		For~$i\in[r]$ let~$\mu_i:[n]_0\to\mathds{R}_{\geq 0}$ be non-increasing, and let~$\mathcal{F}_i\subseteq\ccP([n])$.
		If~$\mathcal{F}_1,\dots, \mathcal{F}_r$ are non-empty~$r$-cross~$t$-intersecting families, then
		\begin{align}
			\sum_{j\in [r]}\mu_j(\mathcal{F}_j)\leq \max\Big\{\mu_{\ell}(\scatfam{n}{a}{t})+\sum_{j\in[r]\setminus\ell}\mu_j(\blockfam{n}{a})\Big\}\,,
		\end{align}
		where the maximum is taken over~$\ell\in[r]$ and~$a\in[t,n]$.
	\end{theorem}
	As before, this bound is attained.
	Note that by taking~$\mu_i\equiv 1$ for every~$i\in [r]$ we obtain the maximum of~$\sum_{i\in[r]}\vert\mathcal{F}_i\vert$ for~$r$-cross~$t$-intersecting families. 
	
	\begin{cor}\label{cor:nonuniform}
		Let~$r\geq 2$,~$n,t\geq 1$ be integers and let~$\mathcal{F}_1,\dots ,\mathcal{F}_r\subseteq \ccP([n])$ be non-empty~$r$-cross~$t$-intersecting families. Then,
		$$\sum_{j\in [r]}\vert\mathcal{F}_j\vert\leq \max_{m\in [t,n]}\Big\{2^{n-m}\sum_{i\in[t,m]}\binom{m}{i}+(r-1)2^{n-m}
		\Big\}\,$$ and this bound is attained.
	\end{cor}
	
	For a further application, note that Theorem~\ref{thm: main2} also provides the maximum for the product measure~$\rho_p$, if~$p\leq 1/2$.
	
	As it turns out, our proofs for Theorems~\ref{thm: main1} and~\ref{thm: main2} are essentially the same and we derive them as special cases of the more general Proposition~\ref{prop: propgen}.
	This proposition even considers the case of~$r$-cross~$t$-intersecting families~$\mathcal{F}_1,\ldots,\mathcal{F}_r$ when some of them satisfy the conditions of Theorem~\ref{thm: main1} and some others satisfy the conditions of Theorem~\ref{thm: main2}.
	
%
%

	\subsection{Idea of the proof}
	
	Our proof is based on what we call \emph{\NIP s} (see Definition \ref{def:nip}).
	Roughly speaking we say that a vertex~$a\in[n]$ is a \NIP~for~$r$-cross~$t$-intersecting families~$\mathcal{F}_1,\dots,\mathcal{F}_r$ if there are sets in the families which ``depend'' on this vertex to fulfil their intersection property.
	For example, if we consider the~$2$-cross~$1$-intersecting families~$\scatfam{n}{2}{1}$ and~$\blockfam{n}{2}$, the vertex~$2$ is a \NIP~because there are pairs of sets that intersect only in~$2$. 
	In this case, $1$ and $2$ are the only \NIP s of these families.
	The idea is to ``decrease'' the maximal \NIP~as long as possible, i.e., replace the presently considered~$r$-cross~$t$-intersecting families by~$r$-cross~$t$-intersecting families whose sum of measures is not smaller but which have a smaller maximal \NIP.
	
	Let~$\mathcal F_1, \dots, \mathcal F_r$ be some~$r$-cross~$t$-intersecting families and let~$a\in[n]$ be their maximal \NIP.
	To construct the new families we first remove all sets that ``depend'' on~$a$ in one family, say~$\mathcal{F}_r$; we call the family of these sets~$\mathcal{F}_r(a)$. 
	Then~$a$ will no longer be a \NIP.
	Potentially, there are some subsets of~$[n]$ which could not be in any of the other families because they would not intersect ``correctly'' with some set in~$\mathcal{F}_r(a)$.
	However, after removing~$\mathcal F_r(a)$ from~$\mathcal{F}_r$ and depending on how such a set relates with~$\mathcal{F}_r\setminus\mathcal{F}_r(a)$, it may be added to one of the other families without breaking the intersection property.
	
	There are some structural properties that follow from~$a$ being the maximal \NIP~and the fact that the families are shifted.
	These will help us to analyse which new sets can actually be added to the families~$\mathcal{F}_1,\dots,\mathcal{F}_{r-1}$ and to prove that in fact the measure of the newly added sets is at least as large as the measure of the removed sets.
	Moreover, this analysis guarantees that the new maximal \NIP~is at most~$a-1$.
	
	We can iterate this construction and decrease the maximal \NIP~in every step.
	This process has to stop at a certain point, and we show that then the resulting families are contained in families with the desired structure (namely~$\scatfam{n}{a}{t}$ and~$\blockfam{n}{a}$).
	
%

	\section{Notation and preliminaries}\label{sec:notation}
	In this section we introduce the notation and some well-known facts about shifting. 
	For a set~$A$ we denote the power set of~$A$ as~$\ccP(A)=\{B : B\subseteq A\}$.
	Let~$j\in\mathds{N}$, then set~$[j]=\{1,\dots ,j\}$,~$[j]_0=[j]\cup\{0\}$, and for~$i\in[j]_0$ set~$[i,j]=\{i,i+1,\dots ,j\}$. For a set with a single element, say $\{i\}$, we sometimes just write $i$.
	Given a set~$A$ we write~$A^{(k)}$ for the set of~$k$-element subsets of~$A$ and similarly~$A^{(\leq k)}$ for the set containing all subsets of~$A$ that are of size at most~$k$. 
	As mentioned in the introduction, for a function~$\mu:[n]_0\to\mathds{R}_{\geq 0}$ and a set~$F\subseteq [n]$ we consider the measure~$F\mapsto\mu(\vert F\vert)$ but omit the vertical lines, i.e., we write~$\mu(F)$ instead of~$\mu(\vert F\vert)$ and refer to~$\mu$ as a measure. 
	Further, for~$\mathcal{F}\subseteq\ccP([n])$ we will write~$\mu(\mathcal{F})=\sum_{F\in\mathcal{F}}\mu(F)$.
	
	
	We now recall the well-known shifting technique.
	For~$F\subseteq[n]$ and~$i,j\in [n]$ we set $$\sigma_{ij}(F)=\begin{cases}
		F\setminus\{j\}\cup\{i\} \text{ if }j\in F\text{ and } i\notin F \\
		F\text{ otherwise }
	\end{cases}\,,$$
	and note that~$\vert \sigma_{ij}(F)\vert=\vert F\vert$.
	Moreover, for a given~$\mathcal{F}\subseteq\ccP([n])$ we define the family~$\sigma_{ij}(\mathcal{F})=\{\sigma_{ij}(F):F\in\mathcal{F}\}\cup\{F\in\mathcal{F}:\sigma_{ij}(F)\in\mathcal{F}\}$ and note that~$\vert\sigma_{ij}(\mathcal{F})\vert=\vert\mathcal{F}\vert$.
	Further, it can easily be checked that if~$\mathcal{F}\subseteq\ccP([n])$ is intersecting, then~$\sigma_{ij}(\mathcal{F})$ is also intersecting.
	We say that~$\mathcal{F}\subseteq\ccP([n])$ is \emph{shifted} if for all~$i,j\in[n]$ with~$i<j$ we have~$\sigma_{ij}(\mathcal{F})=\mathcal{F}$, i.e., for all~$F\in\mathcal{F}$ we have that~$\sigma_{ij}(F)\in\mathcal{F}$.
	By shifting an intersecting family~$\mathcal{F}\subseteq\ccP([n])$ repeatedly, that is, replacing~$\mathcal{F}$ by~$\sigma_{ij}(\mathcal{F})$ repeatedly for all~$i,j\in[n]$ with~$i<j$, we obtain an intersecting family~$\mathcal{F}'$ that is shifted and for which we have~$\vert\mathcal{F}'\vert=\vert\mathcal{F}\vert$ and~$\vert F'^{k}\vert = \vert F^{k} \vert$.
	Thus, to determine the maximal size of an intersecting family, one can restrict themselves to shifted families.
	
	Moreover, for the sake of completeness, we prove the following fact.
	\begin{fact}
		Let~$a,b\in[n]$. 
		If~$\mathcal{F}_1,\dots,\mathcal{F}_r\subseteq\ccP([n])$ are~$r$-cross~$t$-intersecting, then the families~$\sigma_{ab}(\mathcal{F}_1),\ldots,\sigma_{ab}(\mathcal{F}_r)$ are~$r$-cross~$t$-intersecting.
	\end{fact}
	\begin{proof}
		Assume the contrary and let~$F'_{i}\in\sigma_{ab}(\mathcal{F}_i)$ for every~$i\in[r]$ such that~$\vert\bigcap_{i\in[r]} F'_i\vert<t$.
		For every~$i\in[r]$, let~$F_i=F'_i$ if~$F'_i\in\mathcal{F}$.
		If~$F'_i\notin\mathcal{F}$, we know that~$a\in F_i'$ and~$b\notin F_i'$ and we set~$F_i=\sigma_{ba}(F'_i)\in\mathcal{F}_i$.
		Since~$\mathcal{F}_1,\dots,\mathcal{F}_r$ are~$r$-cross~$t$-intersecting, we have~$\vert\bigcap_{i\in[r]}F_i\vert\geq t$ and so there is some~$j\in[r]$ such that~$F_j=\sigma_{ba}(F'_j)\neq F_j'$.
		But then we have~$a\notin F_j$ and, thus,~$a\notin\bigcap_{i\in[r]}F_i$.
		This yields that
		\begin{align}\label{eq:shifting}
			t-1\leq\vert\bigcap_{i\in[r]}F_i\cap([n]\setminus\{a,b\})\vert=\vert\bigcap_{i\in[r]}F'_i\cap([n]\setminus\{a,b\})\vert\,.
		\end{align}
		Note that the assumption~$\vert\bigcap_{i\in[r]}F'_i\vert<t$ tells us that in fact the left side inequality above is an equality.
		This in turn implies that~$b\in\bigcap_{i\in[r]}F_i$.
		
		Our assumption together with~\eqref{eq:shifting} also give some~$\ell\in[r]$ such that~$a\notin F'_{\ell}$.
		Then it follows by definition that~$\sigma_{ab}(F_{\ell})\in\mathcal{F}_{\ell}$ because~$b\in\bigcap_{i\in [r]}F_i$.
		Hence,~$\vert\sigma_{ab}(F_{\ell})\cap\bigcap_{i\in[r]\setminus \ell }F_{i}\vert<t$ contradicts~$\mathcal{F}_1,\dots,\mathcal{F}_r$ being~$r$-cross~$t$-intersecting.
	\end{proof}
	This allows us to restrict ourselves to shifted families when looking for the maximum sum of measures of~$r$-cross~$t$-intersecting families if the measure of a set~$F$ depends only on the size of~$F$.

	\section{Proof of Theorems~\ref{thm: main1} and~\ref{thm: main2}}\label{sec:proofmain}
	We begin by introducing \NIP s which are central to our proofs.
	\begin{dfn}\label{def:nip}
		Let~$\mathcal{F}_1\subseteq\ccP([n]),\dots ,\mathcal{F}_r\subseteq\ccP([n])$ be~$r$-cross~$t$-intersecting families.
		We say~$a \in [n]$ is a \emph{\NIP} of~$\mathcal{F}_1,\dots ,\mathcal{F}_r$ if for all~$j\in[r]$ there is an~$F_j \in \mathcal{F}_j$ such that
		\begin{align}\label{eq:nipwit}
			|[a] \cap\bigcap_{j \in [r]}F_j|= t \qand a \in \bigcap_{j \in [r]}F_j \,.
		\end{align}
	\end{dfn}
	
	The following easy lemma is one of the useful properties of necessary intersection points used together with shifting.
	\begin{lemma}\label{lem:crosscover}
		Let~$\mathcal{F}_1\subseteq\ccP([n]),\dots ,\mathcal{F}_r\subseteq\ccP([n])$ be shifted~$r$-cross~$t$-intersecting families and let~$a$ be their maximal \NIP.
		If~$i\in[r]$,~$F\in\mathcal{F}_i$, and~$F_j\in\mathcal{F}_j$ for~$j\in[r]\setminus i$ are such that~$\vert [a-1]\cap F\cap \bigcap_{j\in[r]\setminus i}F_j\vert<t$, then~$[a-1]\subseteq F\cup\bigcap_{j\in[r]\setminus i}F_j$.
	\end{lemma}
	\begin{proof}
		We will assume that there is a~$b\in [a-1]\setminus(F\cup\bigcup_{j\in[r]\setminus i}F_j)$ and derive a contradiction.
		First let us note that~$a\in F$.
		If~$a\notin F$, we would have~$\vert [a]\cap F\cap \bigcap_{j\in[r]\setminus i}F_j\vert<t$.
		Since~$\mathcal{F}_1,\dots,\mathcal{F}_r$ are~$r$-cross~$t$-intersecting, this would imply that there is a \NIP~larger than~$a$.
		This in turn would be a contradiction to~$a$ being the maximal \NIP~of~$\mathcal{F}_1,\dots,\mathcal{F}_r$ and so we conclude that~$a\in F$.
		
		Further, we know that~$\sigma_{ba}(F)\in\mathcal{F}_i$ since~$\mathcal{F}_i$ is shifted and~$b<a$.
		But then we have~$\vert[a]\cap\sigma_{ba}(F)\cap\bigcap_{j\in[r]\setminus i}F_j\vert<t\,$, which again contradicts~$\mathcal{F}_1,\dots,\mathcal{F}_r$ being~$r$-cross~$t$-intersecting with maximal \NIP~$a$.
	\end{proof}

	Roughly speaking, the proof proceeds by iteratively decreasing the maximal \NIP, i.e., replacing the currently considered families by families with a smaller maximal \NIP.
	In this ``updating'' process we need to be careful with those sets which need~$a$ fulfil the intersection property.
	To make this more precise, we introduce the following notation.
	
	Let~$\mathcal{F}_1\subseteq\ccP([n]),\dots ,\mathcal{F}_r\subseteq\ccP([n])$ be~$r$-cross~$t$-intersecting families and let~$a$ be their maximal \NIP.
	For every~$j\in [r]$ define~$\mathcal{F}_j(a)$ to be the set of all~$F\in\mathcal{F}_j$ for which there exist~$F_{i}\in\mathcal{F}_{i}$ for every~$i\in[r]\setminus j$ such that~\eqref{eq:nipwit} holds.
	We also refer to the sets in~$\mathcal{F}_j(a)$ as the sets in~$\mathcal{F}_j$~\emph{depending on~$a$}. 
	Further, for~$A\subseteq[a-1]$ set~$\mathcal{F}_j(A,a)=\{F\in\mathcal{F}_j(a):F\cap[a-1]=A\}$.
	
	The following lemma is the key of our proof.
	It will allow us to ``push down'' the maximal \NIP~of the families considered in case that we are not already done.
	Since we will prove Theorem~\ref{thm: main1} and Theorem~\ref{thm: main2} simultaneously (by proving Proposition~\ref{prop: propgen}), we phrase this lemma in a general setting.
	The families with indices in~$[r_1]$ are families as in Theorem~\ref{thm: main1} and the remaining families are as in Theorem~\ref{thm: main2}.
	\begin{lemma}\label{lem:update}
		Let~$r,t,n\in\mathds{N}$,~$r_1\in\mathds{N}_0$ with~$r\geq r_1$,~$r\geq 2$, and~$r_1\neq 1$, and let~$a\in[n]$.
		If~$r_1\geq 2$, for~$i\in[r_1]$ let~$k_i\in[n]$ be such that~$n\geq 2\displaystyle\max_{i\in[r_1]}k_i+\underset{i\in[r_1]}{\secmin}k_i-t$, and let~$\mu_i:[n]_0\to\mathds{R}_{\geq 0}$.
		For~$i\in[r_1+1,r]$ set~$k_i=n$ and let~$\mu_i:[n]_0\to\mathds{R}_{\geq 0}$ be non-increasing. For~$i\in[r]$ let~$\mathcal{F}_i\subseteq [n]^{(\leq k_i)}$.
		If~$\mathcal{F}_1,\dots, \mathcal{F}_r$ are shifted~$r$-cross~$t$-intersecting families with maximal \NIP~$a\geq t+1$ such that for all~$i\in[r]$ the family~$\mathcal{F}_i\setminus\mathcal{F}_i(a)$ is non-empty, then
		there are non-empty families~$\mathcal H_1,\dots, \mathcal H_r$ such that
		\begin{enumerate}[label=(\alph*)]
			\item\label{it:setting} for~$i\in[r]$ we have $\mathcal{H}_i\subseteq [n]^{(\leq k_i)}$,
			\item\label{it:rcross} $\mathcal H_1,\dots,\mathcal H_r$ are~$r$-cross~$t$-intersecting with maximal \NIP~at most~$a-1$, and
			\item\label{it:size} $\displaystyle\sum_{j\in [r]}\mu_j (\mathcal H_j)\geq \displaystyle\sum_{j\in [r]}\mu_j (\mathcal F_j) $.
		\end{enumerate}
	\end{lemma}
	\begin{proof}
	
		Roughly speaking, the families~$\mathcal{H}_1,\dots,\mathcal{H}_r$ will be obtained from~$\mathcal{F}_1,\dots,\mathcal{F}_r$ by deleting~$\mathcal{F}_i(a)$ from some of them and adding new sets to the others.
		More precisely, define for every~$i\in[r_1]$ the family
		\begin{align}\label{eq:crossadd}
		\mathcal{F}_i^{\text{add}}=\bigdcup_{k\in[k_i]}\bigdcup_{\substack{A\subseteq[a-1]:\\ \mathcal{F}_i(A,a)^k\neq\emptyset}}\{A\cup T:T\in[a+1,n]^{(k-\vert A\vert)}\}\,.
		\end{align}
		and for~$i\in[r_1+1,r]$ define the family~$\mathcal{F}_i^{\text{add}}=\{F\setminus a:F\in\mathcal{F}_i(a)\}$.
		Next, for~$i\in[r]$ we set~$\mathcal{F}_i^-=\mathcal{F}_i\setminus\mathcal{F}_i(a)$ and~$\mathcal{F}_i^+=\mathcal{F}_i\cup\mathcal{F}_i^{\text{add}}$.
		Note that for all~$i\in[r]$ we have~$\mathcal{F}_i^+,\mathcal{F}_i^-\subseteq[n]^{(\leq k_i)}$ and, hence, they satisfy~\ref{it:setting}.
		
		We aim to show that considering~$\mathcal{F}_i^-$ for some indices and~$\mathcal{F}_j^+$ for the other indices will yield families as desired.
		To this end let us now observe the following claim, ensuring that such a collection will fulfil~\ref{it:rcross}.
		\begin{claim}\label{cl:crossstillint}
			Let~$i\in[r]$.
			\begin{enumerate}
				\item\label{it:crossstillint1} The families~$\mathcal{F}_1^-,\dots,\mathcal{F}_{i-1}^-,\mathcal{F}_i^+,\mathcal{F}_{i+1}^-,\dots,\mathcal{F}_r^-$ are~$r$-cross~$t$-intersecting with maximal \NIP~at most~$a-1$.
				\item\label{it:crossstillint2} The families~$\mathcal{F}_1^+,\dots,\mathcal{F}_{i-1}^+,\mathcal{F}_i^-,\mathcal{F}_{i+1}^+,\dots,\mathcal{F}_r^+$ are~$r$-cross~$t$-intersecting with maximal \NIP~at most~$a-1$.
			\end{enumerate}
		\end{claim}
		\begin{proof}
			\eqref{it:crossstillint1}: Assume the contrary and let~$F_j\in\mathcal{F}_j^-$ for~$j\in[r]\setminus i$ and~$F_i\in\mathcal{F}_i^+$ such that~$\vert [a-1]\cap \bigcap_{j \in [r]}F_j\vert<t$.
			Since~$\mathcal{F}_1,\dots,\mathcal{F}_r$ are~$r$-cross~$t$-intersecting, this means that there is some~$F'\in\mathcal{F}_i(a)$ (potentially~$F'=F_i$) with~$F_i\cap[a-1]=F'\cap[a-1]$.
			But then~$\vert [a-1]\cap F'\cap\bigcap_{j\in[r]\setminus i}F_j\vert<t$, which is a contradiction because~$F_j\in\mathcal{F}_j^-=\mathcal{F}_j\setminus\mathcal{F}_j(a)$.
			
			\eqref{it:crossstillint2}: Assume the contrary and let~$F_j\in\mathcal{F}_j^+$ for~$j\in[r]\setminus i$ and~$F_i\in\mathcal{F}_i^-$ such that~$\vert [a-1]\cap \bigcap_{j \in [r]}F_j\vert<t$.
			Since~$\mathcal{F}_1,\dots,\mathcal{F}_r$ are~$r$-cross~$t$-intersecting, this means that for all~$j\in[r]\setminus i$ there is an~$F'_j\in\mathcal{F}_j(a)$ with~$F_j\cap[a-1]=F'_j\cap[a-1]$.
			But then~$\vert [a-1]\cap F_i\cap\bigcap_{j\in[r]\setminus i}F'_j\vert<t$, which is a contradiction because~$F_i\in\mathcal{F}_i^-=\mathcal{F}_i\setminus\mathcal{F}_i(a)$.
		\end{proof}
		
		Now, let us show that the updated families will still have maximal measure, that is, that~\ref{it:size} holds.
		This essentially follows from the next two claims.
		\begin{claim}\label{cl:crossmeasure}
			For~$i\in[r]$ we have~$\mu_i(\mathcal{F}_i^{\text{add}})\geq\mu_i(\mathcal{F}_i(a))$.
		\end{claim}
		\begin{proof}
			If~$i\in[r_1+1,r]$, note that the definition of~$\mathcal{F}_i^{\text{add}}$ implies an injection~$\varphi:\mathcal{F}_i(a)\to\mathcal{F}_i^{\text{add}}$ with~$\vert \varphi(F)\vert =\vert F\vert-1$.
			Thus, recalling that~$\mu_i$ is non-increasing, the claim is proved.
			
			If~$i\in[r_1]$, we need to work a bit more.
			If~$r_1=0$, there is nothing else to show, so assume that~$r_1\geq 2$.
			First, we want to get an upper bound on~$a$.
			Let~$s$ be the minimal integer such that there is some~$m_*\in[r_1]$ and~$A_*\in[a-1]^{(s)}$ such that~$\mathcal{F}_{m_*}(A_*,a)\neq\emptyset$.
			By definition we know that for~$F\in\mathcal{F}_{m_*}(A_*,a)$ there are~$F_j\in\mathcal{F}_j$ for all~$j\in[r]\setminus m_*$ such that~$\vert [a-1]\cap F\cap\bigcap_{j\in[r]\setminus m_*}F_j\vert<t$.
			Thus, Lemma~\ref{lem:crosscover} yields that~$[a-1]\subseteq F\cup\bigcap_{j\in[r]\setminus m_*}F_j$.
			Since~$\vert F\cap[a-1]\vert=\vert A_*\vert=s$ and~$r_1\geq 2$, this entails~$a\leq s+1+\min_{j\in[r_1]\setminus m_*}k_j-t$.

			To show~$\mu_i(\mathcal{F}_i^{\text{add}})\geq\mu_i(\mathcal{F}_i(a))$ it is enough to show that for all~$k\in[k_i]$ we have~$\vert (\mathcal{F}_i^{\text{add}})^{k}\vert\geq\vert (\mathcal{F}_i(a))^{k}\vert$.
			
			Further, it is easy to see that for all~$k\in[k_i]$ $$(\mathcal{F}_i(a))^{k}\subseteq\bigdcup_{\substack{A\subseteq[a-1]:\\ \mathcal{F}_i(A,a)^k\neq\emptyset}}\{A\cup a\cup T:T\in[a+1,n]^{(k-1-\vert A\vert)}\}\,.$$
			Hence, in view of~\eqref{eq:crossadd}, to show~$\vert (\mathcal{F}_i^{\text{add}})^{k}\vert\geq\vert (\mathcal{F}_i(a))^{k}\vert$ it is enough to show that for every~$A\subseteq[a-1]$ with~$\mathcal{F}_i(A,a)^k\neq\emptyset$ we have~$\binom{n-a}{k-1-\vert A\vert}\leq\binom{n-a}{k-\vert A\vert}$, which in turn holds if~$\frac{n-a}{2}>k-1-\vert A\vert$.
			And indeed, the bounds on~$a$ and~$n$ entail $$\frac{n-a}{2}\geq\frac{n-s-1-\min_{j\in[r_1]\setminus m_*}k_j+t}{2}\geq\frac{2\max_{i\in[r_1]}k_i-s-1}{2}>k-1-\vert A\vert\,.$$
		\end{proof}
		
		Further let us observe the following.
		\begin{claim}\label{cl:crossdisjoint}
			For~$i\in[r]$ we have~$\mathcal{F}_i\cap\mathcal{F}_i^{\text{add}}=\emptyset$.
		\end{claim}
		\begin{proof}
			Assume there is some~$F\in\mathcal{F}_i\cap\mathcal{F}_i^{\text{add}}$.
			Then, because~$F\in\mathcal{F}_i^{\text{add}}$, there is some~$F'\in\mathcal{F}_i(a)$ with~$[a-1]\cap F=[a-1]\cap F'$.
			For~$F'$ on the other hand, there are~$F_j\in\mathcal{F}_j$ for all~$j\in[r]\setminus i$ such that~$\vert [a]\cap F'\cap\bigcap_{j\in[r]\setminus i}F_j\vert =t$ and~$a\in F'\cap\bigcap_{j\in[r]\setminus i}F_j$.
			But since~$F\in\mathcal{F}_i^{\text{add}}$, we know that~$a\notin F$ and thus we have~$\vert [a]\cap F\cap\bigcap_{j\in[r]\setminus i}F_j\vert <t$.
			This gives us a contradiction since~$F\in\mathcal{F}_i$ and~$\mathcal{F}_1,\dots,\mathcal{F}_r$ are~$r$-cross~$t$-intersecting with maximal \NIP~$a$.
		\end{proof}
		
		Finally, we can ``update'' the collection of families.
		If~$\mu_r(\mathcal{F}_r(a))\leq\sum_{i\in[r-1]}\mu_i(\mathcal{F}_i(a))$, we consider the families~$\mathcal{H}_i=\mathcal{F}_i^+$ for~$i\in[r-1]$ and~$\mathcal{H}_r=\mathcal{F}_r^-$.
		Recall that we have~$\mathcal{H}_i\subseteq[n]^{(k_i)}$ for~$i\in[r]$ and that they are non-empty by the condition that~$\mathcal{F}_i\setminus\mathcal{F}_i(a)\neq\emptyset$ for all~$i\in[r]$.
		By Claim~\ref{cl:crossstillint} these families are~$r$-cross~$t$-intersecting with their maximal \NIP~at most~$a-1$ and by Claim~\ref{cl:crossmeasure}, Claim~\ref{cl:crossdisjoint}, and~$\mu_r(\mathcal{F}_i(a))\leq\sum_{i\in[r-1]}\mu_i(\mathcal{F}_i(a))$ we have~$\sum_{i\in[r]}\mu_i(\mathcal{F}_i)\leq\sum_{i\in[r-1]}\mu_i(\mathcal{F}_i^+)+\mu_r(\mathcal{F}_r^-)$.
		Together, this yields~\ref{it:setting}-\ref{it:size} in the conclusion of the lemma.
		
		If~$\mu_r(\mathcal{F}_r(a))\geq\sum_{i\in[r-1]}\mu_i(\mathcal{F}_i(a))$, we consider the families~$\mathcal{F}_1^-,\dots,\mathcal{F}_{r-1}^-,\mathcal{F}_r^+$.
		Similarly as before, it follows that these will satisfy~\ref{it:setting}-\ref{it:size}.
	\end{proof}

	Both Theorem~\ref{thm: main1} and Theorem~\ref{thm: main2} can be obtained from the following more general result by setting~$r_1=r$ and~$r_1=0$ respectively.
	Moreover, this result also provides the maximum of~$\sum_{i\in[r]}\mu_i(\mathcal{F}_i)$ in the case when some of the families and measures satisfy the conditions in Theorem~\ref{thm: main1} and the others satisfy those in Theorem~\ref{thm: main2}.
	\begin{prop}\label{prop: propgen}
		Let~$r,t,n\in\mathds{N}$,~$r_1\in\mathds{N}_0$ with~$r\geq r_1$,~$r\geq 2$, and~$r_1\neq 1$, and let~$a\in[n]$.
		If~$r_1\geq 2$, for~$i\in[r_1]$ let~$k_i\in[n]$ be such that~$n\geq 2\displaystyle\max_{i\in[r_1]}k_i+\underset{i\in[r_1]}{\secmin}k_i-t$, and let~$\mu_i:[n]_0\to\mathds{R}_{\geq 0}$.
		For~$i\in[r_1+1,r]$ set~$k_i=n$ and let~$\mu_i:[n]_0\to\mathds{R}_{\geq 0}$ be non-increasing. For~$i\in[r]$ let~$\mathcal{F}_i\subseteq [n]^{(\leq k_i)}$.
		If~$\mathcal{F}_1,\dots, \mathcal{F}_r$ are non-empty~$r$-cross~$t$-intersecting families with maximal \NIP~at most~$a$, then
		\begin{align}
			\sum_{j\in [r]}\mu_j(\mathcal{F}_j)\leq \max\Big\{\mu_{\ell}(\scatfam{n}{a_*}{t}^{\leq k_{\ell}})+\sum_{j\in[r]\setminus\ell}\mu_j(\blockfam{n}{a_*}^{\leq k_j})\Big\}\,,
		\end{align}
		where the maximum is taken over~$\ell\in[r]$ and~$a_*\in\Big[t,\min\big\{a,\displaystyle\min_{i\in[r]\setminus\ell}k_i\big\}\Big]$.
	\end{prop}
	\begin{proof}
		We perform an induction on~$r$. 
		The beginning is the same for the induction start and the induction step.
		Let all the parameters and~$
		\mu_i$ be given as in the statement of the theorem and note that without restriction~$t\leq \min_{i\in[r]}k_i$.
		Further, let~$\mathcal{F}_1,\dots, \mathcal{F}_r$ be such that
		\begin{enumerate}
			\item\label{it:crosssetting} for~$i\in[r]$ we have $\mathcal{F}_i\subseteq [n]^{(\leq k_i)}$,
			\item\label{it:crossacrossint} they are~$r$-cross~$t$-intersecting with maximal \NIP~at most~$a$,
			\item \label{it:crossmaxsizes} they maximise~$\sum_{j\in[r]}\mu_j(\mathcal{F}_j)$ among all families satisfying~\eqref{it:crosssetting} and~\eqref{it:crossacrossint}, 
			\item\label{it:crossminNIP} their maximal \NIP~is minimal among those families that fulfil~\eqref{it:crosssetting},~\eqref{it:crossacrossint}, and~\eqref{it:crossmaxsizes}.
		\end{enumerate}
		Since the properties~\eqref{it:crosssetting},~\eqref{it:crossacrossint},~\eqref{it:crossmaxsizes}, and~\eqref{it:crossminNIP} are preserved when shifting, we may assume that~$\mathcal F_1,\dots, \mathcal F_r$ are shifted. 
		Denote the maximal \NIP~of~$\mathcal{F}_1,\dots ,\mathcal{F}_r$ by~$a_*$ and observe that if~$a_*=t$, we are done.
		So we assume that~$a_*\geq t+1$. 
		
		First, consider the case that for all~$i\in[r]$ we have that~$\mathcal{F}_i^-\neq\emptyset$.
		Then Lemma~\ref{lem:update} yields families~$\mathcal{H}_1,\dots,\mathcal{H}_r$ satisfying~\eqref{it:crosssetting}-\eqref{it:crossmaxsizes} with a maximal \NIP~smaller than~$a_*$.
		This is a contradiction to the choice of the families (see~\eqref{it:crossminNIP}) and thereby completes the proof of both the induction start and the induction step.
		
		Second, consider the case that for some~$j\in[r]$, without loss of generality~$r$, it holds that~$\mathcal{F}_r\setminus\mathcal{F}_r(a_*)=\emptyset$. 
		That is to say, all sets in $\mathcal F_r$ depend on $a_*$.
		
		Assume that there is a $b\in [a_*-1]$ and $F\in \mathcal F_r$ such that $b\notin F$. 
		As $F_r$ is shifted, we have that $\sigma_{ba_*}(F)\in \mathcal F_r$, but this set does not depend on $a_*$.
		Hence, for every~$F\in\mathcal{F}_r$ we have~$[a_*]\subseteq F$, in other words~$\mathcal{F}_r\subseteq\blockfam{n}{a_*}^{\leq k_r}$.
	
		For~$r=2$ notice that since~$a_*$ is the maximal \NIP, every~$F_1\in\mathcal{F}_1$ has at least~$t$ elements in~$[a_*]$.
		This yields~$\mathcal{F}_1\subseteq\scatfam{n}{a_*}{t}^{\leq k_1}$ and hence $$\mu_1(\mathcal{F}_1)+\mu_2(\mathcal{F}_2)\leq \mu_1(\scatfam{n}{a_*}{t}^{\leq k_1})+\mu_2(\blockfam{n}{a_*}^{\leq k_2})\,,$$ which finishes the proof of the induction start.
		
		For~$r\geq 3$ observe that the families~$\mathcal{F}_1,\dots,\mathcal{F}_{r-1}$ are~$(r-1)$-cross~$t$-intersecting families with maximal \NIP~at most $a_*$ which maximise~$\sum_{j\in[r-1]}\mu_j(\mathcal{F}_j)$ (among all~$(r-1)$-cross~$t$-intersecting families~$\mathcal{G}_i\subseteq[n]^{(k_i)}$ with maximal \NIP~at most $a_*$).
		Thus, the induction hypothesis implies that there is an~$\ell\in[r-1]$ and an~$a_{**}\in[a_*]$ such that 
		$$\sum_{j\in[r-1]}\mu_j(\mathcal{F}_j)\leq
		\mu_{\ell}(\scatfam{n}{a_{**}}{t}^{\leq k_{\ell}})
		+
		\sum_{j\in[r-1]\setminus\ell}\mu_j (\blockfam{n}{a_{**}}^{\leq k_j})\,.$$
		Since~$\mathcal{F}_r\subseteq\blockfam{n}{a_*}^{\leq k_r}\subseteq\blockfam{n}{a_{**}}^{\leq k_r}$, this entails $$\sum_{j\in[r]}\mu_j(\mathcal{F}_j)\leq\mu_{\ell}(\scatfam{n}{a_{**}}{t}^{\leq k_{\ell}})+\sum_{j\in[r]\setminus\ell}\mu_j(\blockfam{n}{a_{**}}^{\leq k_j})\,$$
		which finishes the induction step.
	\end{proof}
	
	\section{Concluding remarks}\label{sec: conrem}
	Observe that the maxima in our results are attained for different~$i$ (and~$\ell$), depending on the measures and~$r$,~$t$, and~$n$.
	However, we remark the following.
	\begin{remark}\label{rem:larger}
		For given~$t, n, k \in \mathds{N}$ and a measure~$\mu$ there is an~$r_0$ such that if~$r\geq r_0$, the maximum in Theorem~\ref{thm: main1} and Theorem~\ref{thm: main2} is always attained for~$i=t$ if~$\mu=\mu_j$ (and~$k_j=k$) for all~$j\in[r]$.
	\end{remark}
	
	One can also ask for the maximum of the product of sizes or, more generally, the product of measures of~$r$-cross~$t$-intersecting families, instead of the sum. 
	More precisely, for given measures~$\mu_1,\ldots,\mu_r$ find the maximum possible value of \setlength{\abovedisplayskip}{0pt}
	\begin{align}\label{prob:product}
		\prod_{i\in [r]}\mu_i(\mathcal{F}_i) \,
	\end{align}
	for~$\mathcal{F}_1,\ldots ,\mathcal{F}_r$ being $r$-cross $t$-intersecting families.
	
	There are some partial results concerning this problem (\cite{FranklKupavskii2, Borgcrosstintprod, Pyber,Tokushigercross}).
	Frankl and Tokushige~\cite{FranklTokushigeproduct} determined the maximal product of the sizes of~$r$-cross~$1$-intersecting families.
	In~\cite{Borgcrosstintprod}, Borg determined the maximum of~\eqref{prob:product} for~$r=2$ and measures with certain properties (which include the product measure, the uniform measure, and the constant measure).
	It is well known that for~$a_1,\ldots,a_r\in\mathds{R}_{\geq 0}$ with~$\sum_{i\in[r]}a_i\leq a$ the product~$\prod_{i\in[r]}a_i$ is maximised if~$a_i=\frac{a}{r}$ for all~$i\in[r]$. 
	Therefore, considering Remark~\ref{rem:larger}, given~$n$, measures~$\mu_i=\mu$ (and~$k_i=k$) with~$\mu$ (and~$k$ and~$n$) satisfying the conditions in Theorem~\ref{thm: main1} or Theorem~\ref{thm: main2}, there is an~$r_0$ such that for~$r\geq r_0$ these theorems actually also yield that the maximum of~\eqref{prob:product} is~$(\mu(\blockfam{n}{t}^{\leq k}))^r$.
	This particularly includes the product measure, the uniform measure, and the constant measure, and solves a few instances of the Problems~12.10 and~12.11, and of the Conjectures~12.12 and~12.13 posed by Frankl and Tokushige in~\cite{FranklTokushige}.

\end{document}